\documentclass[12pt]{article}
\usepackage{amsthm,amsmath,amssymb,latexsym,amscd}
\newcommand{\g}{\frak{g}}
\newcommand{\h}{\frak{h}}
\renewcommand{\L}{\frak{L}}

\newcommand{\s}{\mathbf{s}}

\newcommand{\A}{\mathcal{A}}
\newcommand{\F}{\mathcal{F}}
\newcommand{\D}{\mathcal{D}}

\renewcommand{\O}{\mathcal{O}}
\renewcommand{\P}{\mathcal{P}}
\newcommand{\Q}{\mathcal{Q}}
\renewcommand{\S}{\mathcal{S}}
\newcommand{\T}{\mathcal{T}}

\newcommand{\Hom}{\mathrm{Hom}}

\newcommand{\Ker}{\mathrm{Ker}}
\renewcommand{\Im}{\mathrm{Im}}

\newcommand{\DC}{\mathcal{DC}}
\newcommand{\DR}{\mathcal{DR}}

\newcommand{\Com}{\mathcal{C}om}

\newcommand{\Lie}{\mathcal{L}ie}
\newcommand{\LL}{\mathcal{L}\mathcal{L}}
\newcommand{\Leib}{\mathcal{L}eib}

\renewcommand{\d}{\mathbf{d}}
\newcommand{\p}{\prime}

\newcommand{\pa}{\partial}
\renewcommand{\c}{\circ}
\newcommand{\ot}{\otimes}
\newcommand{\ti}{\tilde}

\newtheorem{definition}{Definition}[section]
\newtheorem{lemma}[definition]{Lemma}
\newtheorem{claim}[definition]{Claim}
\newtheorem{proposition}[definition]{Proposition}
\newtheorem{theorem}[definition]{Theorem}
\newtheorem{corollary}[definition]{Corollary}
\newtheorem{remark}[definition]{Remark}

\date{}

\begin{document}

\title{
Derived bracket construction and 
anti-cyclic subcomplex of Leibniz (co)homology complex
}
\author{K. UCHINO}
\maketitle
\abstract
{
An arbitrary Leibniz algebra can be embedded in
a differential graded Lie algebra
via the derived bracket construction.
Such an embedding is called a derived bracket representation.
We will construct the universal version
of the derived bracket representation,
prove that the principal part of the
target dg Lie algebra defines a subcomplex of Leibniz
(co)homology complex
and that the existence of the subcomplex
is a reflection of the anti-cyclicity of
the Leibniz operad.
}
\section{Introduction}

Leibniz algebras are vector spaces equipped with
binary bracket products satisfying the Leibniz
identity.
The notion of Leibniz algebra was introduced
by Jean-Louis Loday,
motivated by the study of algebraic K-theory.
Hence the Leibniz algebras are sometimes called
the Loday algebras.
Today it is widely known that
Leibniz algebras arise in many areas of mathematics
not only K-theory.
To study Leibniz algebras geometrically
the derived bracket construction
of Kosmann-Schwarzbach \cite{Kos1}
is an effectual method.
If $(\h,D)$ is a differential graded Lie algebra
(shortly, dg Lie algebra), then the derived bracket,
$$
[x,y]:=(Dx,y),
\eqno{(0)}
$$
is an odd Leibniz bracket on $\h$,
where $(.,.)$ is the Lie bracket on $\h$ and $x,y\in\h$.
The derived bracketing is a method of
constructing a Leibniz algebra
and one can think that
the deriving differential $D$
has almost every information about the Leibniz bracket.
It is known that the converse also holds.
An arbitrary Leibniz algebra can be embedded
in a dg Lie algebra via the derived bracket construction.
Namely, given a Leibniz algebra $\g$,
there exists a dg Lie algebra $(\h,D)$
which includes $\g$ (precisely $\g[1]$)
and the Leibniz bracket on $\g$ or $\g[1]$
is expressed as the derived bracket on $\h$.
Such an embedding is considered to be
a kind of representation of Leibniz algebra.
So we call this type of representation
a {\em derived bracket representation}.
To find a useful representation is an interesting problem.
For example, the symplectic realization for
Courant algebroids is a derived bracket representation
(for the details see Roytenberg \cite{Roy}).\\
\indent
The first aim of this note is to
construct the universal derived bracket representation.
The derived bracket construction is regarded as
a functor, which is denoted by $\DC$,
from the category of dg Lie algebras
to the one of odd-Leibniz algebras.
The universal representation is defined as
the adjoint functor of the derived bracket construction.
$$
\Hom_{dgLie}(\DR\g[1],\h)
\cong\Hom_{Leib}(\g[1],\DC\h),
$$
where $\DR(-)$ is the functor of
the universal representation.
The second aim is to show that
the principal part of the dg Lie algebra $\DR\g[1]$
is a subcomplex of Loday's complex over $\g$
(the chain complex computing the Leibniz homology group).
Since the Loday complex is a consequence
of the bar construction,
the relation between the derived bracket construction
and the bar construction becomes clear.
The existence of the subcomplex
is closely related with the anti-cyclicity
of the Leibniz operad.
In \cite{Chap} Chapoton proved that
the Leibniz operad is anti-cyclic
(See Section 2.2 below for the details).
In general, if an operad is cyclic (not anti-cyclic),
then the (co)homology complex of the operad-algebra
can be reduced by its symmetry
and the cyclic (co)homology group is defined
(cf. Getzler-Kapranov \cite{GetKap}).
Although the Leibniz operad is not cyclic,
because it is still anti-cyclic,
there exists a subcomplex or quotient complex.
Our subcomplex is exactly that.\\
\indent
Our method of constructing the (sub)complex is not
bar-construction, but the derived bracket construction
or representation.
An advantage of adopting the derived bracket theory,
it is not necessary to use the Koszul duality theory.
To compute the Koszul dual of Leibniz operad
(Zinbiel operad) is not easy by comparison with
the cases of the associative operad and the Lie operad.
By using the derived bracket theory,
one can avoid this problem.
\medskip\\
\noindent
The paper is organized as follows:\\
Section 2 is Preliminaries. We recall
some basic properties of Leibniz algebras
and derived bracket construction (of operadic).\\
In Section 3,
we will construct
the derived bracket representation of universal.
To construct the universal representation
an operad theory will be used.
We will see that
if $\g$ is a Lie algebra as a commutative Leibniz algebra,
then the second homology group of $\DR\g[1]$
is equal to the space of formal 0-forms,
$\Omega^{0}\g$, introduced by Kontsevich \cite{Kont}.
According to Kontsevich,
$\Omega^{0}\g$ is the target space
of the universal invariant bilinear form.
Invariant bilinear forms in the category of Lie algebras
are symmetric pairings satisfying the well-known condition,
$$
\langle{x,[y,z]}\rangle=\langle{[x,y],z}\rangle.
$$
It is well-known that invariant bilinear forms
are induced via the derived bracket construction
(in the case of Drinfeld double by Kosmann-Schwarzbach
\cite{Kos0} and in genral case by Roytenberg \cite{Roy}).
Our result provides the universal version
of the previous studies.\\
In Section 4.1, we will prove that the principal part
of $\DR\g[1]$ is a subcomplex of the Loday complex over $\g$
and the deriving differential on $\DR\g[1]$ is equal to
the boundary map of Loday.
In 4.2,
we study the cohomology counter part of $\DR\g[1]$.
We will introduce the notion of anti-cyclic cochain
for Leibniz algebras.
The anti-cyclic cochains are defined as
the linear functions satisfying a symmetry
induced from the anti-cyclicity of
Leibniz operad, like the cyclic cochains for associative
algebras satisfy
$\varphi(a_{0},...,a_{n})
=(-1)^{n}\varphi(a_{n},a_{0},...,a_{n-1})$
(cf. Connes \cite{Connes}.)
The symmetry that the anti-cyclic cochain satisfies
is more complicated, for instance,
$A(x_{0},x_{1},x_{2})$ is an anti-cyclic 2-cochain
if and only if
\begin{eqnarray*}
A(x_{0},x_{1},x_{2})&=&A(x_{0},x_{2},x_{1}),\\
A(x_{0},x_{1},x_{2})+A(x_{2},x_{0},x_{1})
+A(x_{1},x_{2},x_{0})&=&0.
\end{eqnarray*}
We will prove that the set of anti-cyclic cochains
is a subcomplex of Loday-Phirashviri complex over $\g$
and that the coboundary map of Loday-Phirashviri
is on the subcomplex
the dual of the deriving differential on $\DR\g[1]$.\\
In Section 5, we will study a tensor expression of
the anti-cyclic cochains.
\medskip\\
\noindent
\textbf{Acknowledgement}.
Many parts of this paper were written at
the University of Luxembourg, when the author
was invited by Professor Nobel Poncin.
I would like to thank Professor Poncin,
Ms Katharina Heil, Ms Elodie Reyter
and many staffs of UL
for their kind supports and encouragements.
\section{Preliminaries}
\subsection{Leibniz algebras}

(Left-)Leibniz algebras
are by definition vector spaces $\g$
equipped with binary brackets $[.,.]$
satisfying the (left-)Leibniz identity,
$$
[x_{1},[x_{2},x_{3}]]
=[[x_{1},x_{2}],x_{3}]+[x_{2},[x_{1},x_{3}]],
$$
where $x_{1},x_{2},x_{3}\in\g$.
In the following,
we usually suppose that
the degree of the space $\g$ is homogeneously zero.
Let $\g[1]$ be the shifted space of $\g$,
where the degree of any element $x\in\g[1]$ is $-1$.
The space $\g[1]$ becomes an odd Leibniz algebra
and the degree of the bracket is $+1$.
\medskip\\
\indent
We recall a basic property
of Leibniz algebra, which will be used in the next section.
Let $\g$ be a Leibniz algebra
and let $I$ the space consisting of
the symmetric brackets,
$$
I:=\{[x_{1},x_{2}]+[x_{2},x_{1}],
x_{1},x_{2}\in\g\}.
$$
Then $I$ becomes an ideal of $\g$,
in particular, $[I,\g]=0$.
The quotient space $\g_{Lie}:=\g/I$ becomes a Lie algebra
and this projection $p:\g\to\g_{Lie}$ (so-called Liezation)
is universal, i.e., for any Lie algebra $\h$,
an arbitrary Leibniz algebra morphism
$f:\g\to\h$ factors through $p$,
$f=\psi\c p$,
where $\psi:\g_{Lie}\to\h$ is the Lie homomorphism
corresponding to $f$.
\subsection{Anti-invariant 2-forms for Leibniz algebras}
In this section we suppose that $\g$ is a finite
dimensional Leibniz algebra.
In \cite{Chap} Chapoton proved that
the operad of Leibniz algebras is anti-cyclic.
This means that invariant 2-forms, $\omega(.,.)$,
in the category of Leibniz algebras are anti-symmetric,
i.e., $\omega(x_{1},x_{2})=-\omega(x_{2},x_{1})$
and the invariant condition for $\omega$
is defined by the following two formulas.
\begin{eqnarray}
\label{A1}\omega(x_{1},[x_{2},x_{3}])&=&-\omega([x_{2},x_{1}],x_{3}),\\
\label{A2}\omega(x_{1},[x_{2},x_{3}])&=&\omega([x_{1},x_{3}]+[x_{3},x_{1}],x_{2}).
\end{eqnarray}
Chapoton's original formula has been
defined in the right-version
and we will use the left-version above.
If the 2-form is non-degenerate, then it is a symplectic
structure. Hence we denote it by $\omega$.
Thanks to (\ref{A1})-(\ref{A2}), one can immediately
write-down the coadjoint action in the category of
Leibniz algebras.
\begin{definition}[coadjoint action]
Given a finite dimensional Leibniz algebra $\g$
and its dual space $\g^{*}$,
the coadjoint action of $\g$ to $\g^{*}$ is
by definition,
\begin{eqnarray}
\label{CA1}\langle x_{1},[x_{2},a]\rangle&:=&-\langle[x_{2},x_{1}],a\rangle,\\
\label{CA2}\langle x_{1},[a,x_{2}]\rangle&:=&\langle[x_{1},x_{2}]+[x_{2},x_{1}],a\rangle,
\end{eqnarray}
where $x_{1},x_{2}\in\g$,
$a\in\g^{*}$ and $\langle.,.\rangle$ the natural pairing
betweem $\g$ and $\g^{*}$.
\end{definition}
It is easy to see that the representation
(\ref{CA1})-(\ref{CA2}) is Leibniz, i.e.,
\begin{eqnarray*}
{}[x_{1},[x_{2},a]]&=&[[x_{1},x_{2}],a]+[x_{2},[x_{1},a]],\\
{}[x_{1},[a,x_{2}]]&=&[[x_{1},a],x_{2}]+[a,[x_{1},x_{2}]],\\
{}[a,[x_{1},x_{2}]]&=&[[a,x_{1}],x_{2}]+[x_{1},[a,x_{2}]].
\end{eqnarray*}
Therefore, the semi-direct product,
$\g\ltimes\g^{*}=\g\oplus\g^{*}$,
becomes a Leibniz algebra, whose bracket is
$$
[x_{1}\oplus a_{1},x_{2}\oplus a_{2}]
:=[x_{1},x_{2}]\oplus[x_{1},a_{2}]+[a_{1},x_{2}].
$$
The double space has a canonical symplectic structure,
\begin{equation}\label{defomega}
\omega(x_{1}+a_{1},x_{2}+a_{2}):=
\langle x_{1},a_{2}\rangle-\langle x_{2},a_{1}\rangle.
\end{equation}
The pair $(\g\ltimes\g^{*},\omega)$ satisfies
(\ref{A1})-(\ref{A2}).
\subsection
{Derived bracket construction of operadic}
Let $(\h,(.,.),d)$ be a dg Lie-algebra.
We suppose
that the degree of the Lie bracket is $|(.,.)|=0$
and the one of the differential is $|d|:=+1$.
Define an odd bracket by
\begin{equation}\label{defderived}
[x_{1},x_{2}]:=(dx_{1},x_{2}),
\end{equation}
which is called a binary derived bracket
or derived bracket for short (\cite{Kos1}).
Then $(\h,[.,.]=(d.,.))$
becomes an odd-Leibniz algebra.
\begin{remark}[sign]
Althought in \cite{Kos1}
the derived bracket has been defined as
$(-1)^{|x_{1}|+1}(dx_{1},x_{2})$,
because (\ref{defderived}) has a good compatibility
with operad theory, we will use (\ref{defderived})
as a definition of the derived bracket.
\end{remark}

\begin{definition}[\cite{Uchi}]
Lie-Leibniz algebras are by definition graded spaces
with even-Lie brackets $(.,.)$ and odd-Leibniz brackets
$[.,.]$ satisfying two extra identities,
\begin{eqnarray}
\label{LL1}[x_{1},(x_{2},x_{3})]&=&
([x_{1},x_{2}],x_{3})+(x_{2},[x_{1},x_{3}])\\
\label{LL2}[(x_{1},x_{2}),x_{3}]&=&
([x_{1},x_{2}]-[x_{2},x_{1}],x_{3}),
\end{eqnarray}
where we put $|x_{1}|=|x_{2}|=|x_{3}|:=0$ simply.
\end{definition}
The derived bracket Leibniz algebra $(\h,(.,.),[.,.]=(d.,.))$
is the model of the Lie-Leibniz algebra.
When $|x_{1}|=|x_{2}|=|x_{3}|:=1$,
since the Lie bracket is commutative,
the second condition (\ref{LL2}) has
the following form
$$
[(x_{1},x_{2}),x_{3}]=([x_{1},x_{2}]+[x_{2},x_{1}],x_{3}).
$$
In the following, we put $|(.,.)|:=0$ and $|[.,.]|:=+1$.
\begin{remark}[Jacobi identity
and derived bracket construction]
Many identities that the derived bracket satisfies
are consequences of the Jacobi identity of
the original Lie bracket. However
(\ref{LL2}) is completely
independent from the Jacobi identity.
It is a consequence of the derivation rule $d(x_{1},x_{2})=(dx_{1},x_{2})+(x_{1},dx_{2})$.
\end{remark}
\indent
We recall three known propositions,
which will be used in the next section
to prove the key-lemma of this note.
Before that, we briefly recall algebraic operads.
A collection $\P=(\P(n))$ consisting of
modules $\P(n)$ over the symmetric group
$S_{n}$ is called an $\S$-module.
Given an $\S$-module $\P$, a functor,
so-called Schur functor, is defined by
$$
F_{\P}V:=
\bigoplus_{n\in\mathbb{N}}\P(n)\ot_{S_{n}}V^{\ot n},
$$
where $V$ is a vector space.
Given two $\S$-modules $\P$ and $\Q$,
a tensor product $\P\odot\Q$ is defined
by $F_{\P}F_{\Q}:=F_{\P\odot\Q}$,
more explicitly,
\begin{equation}\label{defodot}
(\P\odot\Q)(n):=\bigoplus_{l_{1}+\cdots+l_{m}=n}
\P(m)\ot_{S_{m}}
(\Q(l_{1}),...,\Q(l_{m}))
\ot_{(S_{l_{1}},...,S_{l_{m}})}S_{n}.
\end{equation}
The $\S$-module $\P$ is called
an operad, if $F_{\P}$ is a triple
(cf. MacLane \cite{Mac}),
namely, if there exists a morphism
(natural transformation),
$F_{\P}F_{\P}\to F_{\P}$,
and if by this operation $F_{\P}$ becomes a unital
associative monoid.
When $\P$ is an operad, the notion of $\P$-algebra
is defined.
If $\A$ is a $\P$-algebra, then the $\P$-algebra
product on $\A$ is defined as a map of
\begin{equation}\label{defgamma}
\gamma:F_{\P}\A\to\A.
\end{equation}
In particular, $\A=F_{\P}V$ is the free $\P$-algebra.
The free operad over an $\S$-module $\Q$,
$\T\Q$, is the free algebra
in the category of operads and an algebraic operad $\P$
is expressed as a quotient operad
of the free operad $\P:=\T\Q/(R)$,
where
$R$ is the relation of $\P$
and $(R)$ is the generated operadic ideal.
If the free operad $\T\Q$ is generated by $\Q(2)$
and if $R$ is a sub $S_{3}$-module of $(\T\Q)(3)$,
then the quotient operad $\P:=\T\Q/(R)$
is called a binary quadratic operad
and $R$ is called the quadratic relation.
For example, the Lie operad (the operad of Lie algebras)
is a binary quadratic operad over $sgn_{2}$,
$$
\Lie:=\T(sgn_{2})/(R_{\Lie}),
$$
where
$sgn_{2}$ is the sign representation of $S_{2}$
and $R_{Lie}$ is the Jacobi identity.
As a result, $\Lie(2)=sgn_{2}$ and the base of $\Lie(2)$
is identified with the universal Lie bracket,
which is denoted by $(1,2)$.
Then $R_{Lie}$ is identified with the space
generated by the Jacobiator,
$$
R_{Lie}=<(1,(2,3))+(3,(1,2))+(2,(3,1))>.
$$
The Leibniz operad, $\Leib$,
is also defined by the same manner.
The parity shift for operad, $\P\mapsto\s\P$,
is defined by $(\s\P)(n):=\P(n)[n-1]\ot sgn_{n}$
for each $n$. If $|\P|=even$, then
$\s\P$-algebras are odd-$\P$-algebras.
\medskip\\
\indent
Let us denote by $\LL$ the operad of Lie-Leibniz algebras,
which is a binary quadratic operad
$$
\LL:=\T\big(\Lie(2)\oplus\s\Leib(2)\big)/(R_{\LL}),
$$
where $R_{\LL}$ is the quadratic relation of $\LL$.
It is obvious that $\LL=(\LL^{i})$
is a graded operad, in particular,
$\LL^{0}=\Lie$ and $\LL^{top}=\s\Leib$.
\begin{proposition}[\cite{Uchi}]\label{keyprop1}
$\LL=\Lie\ot\D:=(\Lie(n)\ot\D(n))$, where
$\D$ is a graded operad defined as follows.
\end{proposition}
To introduce $\D$
we use the following expression of $\Com$
(the operad of commutative associative algebras).
\begin{eqnarray*}
\Com(2)&=&<1\ot 1>\\
\Com(3)&=&<1\ot 1\ot 1>\\
\cdots &=&\cdots\\
\Com(n)&=&<1^{\ot n}>.
\end{eqnarray*}
Let $d$ be a formal 1-ary operator of degree $+1$
and $1\ot 1$ be the generator of $\Com$.
We consider a quadratic operad over $d$ and $1\ot 1$.
$$
\O:=\T(d,1\ot 1)/(R_{\O}),
$$
where $R_{\O}$ is a quadratic relation
generated by
\begin{eqnarray*}
(1\ot 1)\c_{1}(1\ot 1)&=&1\ot 1\ot 1=(1\ot 1)\c_{2}(1\ot 1),\\
d(1\ot 1)&=&d\ot 1+1\ot d,\\
dd&=&0.
\end{eqnarray*}
Namely, $d$ is a differential in $\Com$.
Obviously, $\O$ is a graded operad, $\O=\{\O^{i}\}$,
whose degree is the number of $d$.
For each $n$,
$$
\O(n)=\O^{0}(n)\oplus\O^{1}(n)\oplus\cdots\oplus\O^{n}(n).
$$
and $\O^{0}=\Com$.
There is no $\O^{n+1}(n)$ because $dd=0$.
The operad $\D$ is defined as a suboperad of $\O$:
\begin{definition}[\cite{Uchi}]
For each $n$,
$$
\D(n):=\O^{0}(n)\oplus\O^{1}(n)\oplus\cdots\oplus\O^{n-1}(n).
$$
\end{definition}
This operad is expressed as follows,
\begin{eqnarray*}
\D^{1}(2)&=&<d\ot 1, 1\ot d>,\\
\D^{1}(3)&=&<d\ot 1\ot 1,1\ot d\ot 1,1\ot 1\ot d>,\\
\D^{2}(3)&=&<d\ot d\ot 1,d\ot 1\ot d,1\ot d\ot d>,\\
\cdots &=& \cdots.
\end{eqnarray*}
The elements in $\LL=\Lie\ot\D$ are identified with
formal derived brackets, for instance,
\begin{eqnarray*}
(1,2)\ot(d\ot 1)&=&(d1,2),\\
(1,(2,3))\ot(1\ot d\ot 1)&=&(1,(d2,3)),
\end{eqnarray*}
where $(1,2)$ is the Lie bracket in $\Lie(2)$.
Hence the functor $(-)\ot\D$ is
considered to be a derived bracket
construction of operadic.
\begin{proposition}[\cite{Uchi}]\label{keyprop2}
For each $n$,
$(\LL(n),\delta)$ is complex.
$$
\begin{CD}
\Lie(n)=\LL^{0}(n)@>\delta>>\LL^{1}(n)@>\delta>>\cdots
@>\delta>>\LL^{n-1}(n)=\s\Leib(n).
\end{CD}
$$
\end{proposition}
The differential of the proposition is defined as follows.
For each $n$, $(\D(n),d)$ is clearly a complex,
whose differential is defined by
$$
d(x_{1}\ot\cdots\ot x_{n}):=\sum_{1\le{i}\le n}
(-1)^{|x_{1}|+\cdots+|x_{i-1}|}
(x_{1}\ot\cdots\ot x_{i-1}\ot dx_{i}\ot \cdots \ot x_{n}),
$$
where $x_{j}\in\{1,d\}$.
Since $\LL=\Lie\ot\D$,
we obtain a differential on $\LL$,
\begin{equation}\label{deflied}
\delta:=\Lie\ot d.
\end{equation}
We should remark that $\LL$ is not dg-operad.
\begin{proposition}[\cite{Uchi}]\label{keyprop3}
The Lie-Leibniz identity
(\ref{LL1})-(\ref{LL2}) is a distributive law in the sense
of Markl \cite{Markl}.
Therefore,
the operad $\LL$ is decomposed into $\Lie$ and $\s\Leib$.
$$
\LL=\Lie\odot\s\Leib,
$$
where
$\odot$ is the tensor product in the category of
$\S$-modules, cf., (\ref{defodot}).
\end{proposition}
As a corollary of this proposition, we have
\begin{corollary}[\cite{Uchi}]\label{lieleibcoro}
Let $\g$ be a Leibniz algebra
and let $\g[1]$ the shifted odd-Leibniz algebra.
Then the free-Lie algebra over $\g[1]$,
$F_{Lie}\g[1]$, is a Lie-Leibniz algebra.
When $\g$ is free, the Lie-Leibniz algebra is also free.
\end{corollary}
We will study the Lie-Leibniz algebra $F_{Lie}\g[1]$
in the next section.

\section{Derived bracket representation}

It is known that an arbitrary Leibniz algebra
can be embedded in a dg Lie algebra,
via the derived bracket construction.
There are some methods of proving this proposition.
For instance,
by extending the result in Grabowski et al \cite{GKP},
by the universal method that we will introduce
in the following.
\begin{definition}
Let $\g$ be a Leibniz algebra,
let $\g[1]$ be the shifted odd-Leibniz algebra
and let $(\h,(d.,.))$ the derived bracket Leibniz algebra.
A momorphism of Leibniz algebra
(not necessarily embedding), $\g[1]\to\h$,
is called a derived bracket representation.
\end{definition}
The aim of this section is
to construct the universal representation.
\begin{lemma}
If $\g$ is a Leibniz algebra, then
the graded space $F_{Lie}\g[1]$ is a chain complex,
$$
\begin{CD}
\cdots @>{\d}>>F_{Lie}^{3}\g[1]@>{\d}>>
F_{Lie}^{2}\g[1]@>{\d}>>\g[1]@>{0}>>0.
\end{CD}
$$
\end{lemma}
\begin{proof}
To define $\d$ we use
Propositions \ref{keyprop2} and \ref{keyprop3}.
The differential is defined as
a composition map of $\delta$ and $\gamma$,
where $\delta$ is the differential defined
in (\ref{deflied}) and $\gamma$
is the Leibniz product in (\ref{defgamma}).
\begin{multline*}
\d:F_{Lie}^{n}\g[1]=
\Lie(n)\ot_{S_{n}}\g[1]^{\ot n}
\overset{\delta\ot 1}{\longrightarrow}
\LL^{1}(n)\ot_{S_{n}}\g[1]^{\ot n}=\\
=\Lie(n-1)\odot\s\Leib(2)\ot_{S_{n}}\g[1]^{\ot n}
\overset{\Lie\odot\gamma}{\longrightarrow}
\Lie(n-1)\ot_{S_{n}}\g[1]^{\ot n-1}
=F_{Lie}^{n-1}\g[1].
\end{multline*}
In above sequence, the part of $\gamma$
is more precisely expressed as follows.
\begin{multline*}
\LL^{1}(n)\ot_{S_{n}}\g[1]^{\ot n}=\\
=\bigoplus_{i+j+1=n-1}\Lie(n-1)\ot_{S_{n-1}}
(1^{\ot i},\s\Leib(2),1^{\ot j})
\ot_{(1^{\ot i},S_{2},1^{\ot j})}\ot S_{n}\ot_{S_{n}}
\g[1]^{\ot n}\\
=\bigoplus_{i+j+1=n-1}\Lie(n-1)\ot_{S_{n-1}}
(1^{\ot i},\s\Leib(2),1^{\ot j})
\ot_{(1^{\ot i},S_{2},1^{\ot j})}\ot\g[1]^{\ot n}\\
\end{multline*}
and
$$
\Lie\odot\gamma=
\bigoplus_{i+j+1=n-1}
\Lie(n-1)\ot 1^{\ot i}\ot\gamma\ot 1^{\ot j}.
$$
From the odd Leibniz identity on $\g[1]$,
$\d\d=0$ holds.
\end{proof}
More explicitly $\d$ is computed as follows.
Denote the right normalized bracket\\
$(x_{1},(x_{2},(x_{3},....,x_{n})))$ by simply
$\{x_{1},...,x_{n}\}$.
\begin{multline}\label{explid1}
\d\{x_{1},...,x_{n+1}\}=
\sum_{\substack{i<j \\ i\le n}}
(-1)^{i-1}\{x_{1},...,x_{i}^{\vee},...,[x_{i},x_{j}],x_{j+1},...,x_{n+1}\}+\\
+(-1)^{n-1}\{x_{1},...,x_{n-1},[x_{n+1},x_{n}]\},
\end{multline}
or equivalently,
\begin{multline}\label{explid2}
=\sum_{\substack{i<j \\ i\le n-1}}
(-1)^{i-1}\{x_{1},...,x_{i}^{\vee},...,[x_{i},x_{j}],x_{j+1},...,x_{n+1}\}+\\
+(-1)^{n-1}\{x_{1},...,x_{n-1},[x_{n},x_{n+1}]+[x_{n+1},x_{n}]\}.
\end{multline}
For example,
\begin{eqnarray}
\label{fderirule}
\d(x_{1},x_{2})&=&[x_{1},x_{2}]+[x_{2},x_{1}],\\
\label{fderirule2}
\d(x_{1},(x_{2},x_{3}))&=&([x_{1},x_{2}],x_{3})+
(x_{2},[x_{1},x_{3}])-(x_{1},[x_{2},x_{3}]+[x_{3},x_{2}]).
\end{eqnarray}
We should remark that $(F_{Lie}\g[1],\d)$
is not dg Lie algebra.
Although $\d$ is not derivation,
for any $\alpha_{1},\alpha_{2}\in F_{Lie}^{\ge 2}\g_{Lie}[1]$, it still satisfies the rule of derivation,
\begin{eqnarray*}
\d(\alpha_{1},\alpha_{2})&=&
(\d\alpha_{1},\alpha_{2})+
(-1)^{|\alpha_{1}|}(\alpha_{1},\d\alpha_{2})\\
&=&[\alpha_{1},\alpha_{2}]-
(-1)^{|\alpha_{1}||\alpha_{2}|}
[\alpha_{2},\alpha_{1}],
\end{eqnarray*}
where $[\alpha_{1},\alpha_{2}]$ is the Lie-Leibniz bracket
(recall Corollary \ref{lieleibcoro}).
Hence one can think that $\d$ is an {\em almost derivation}
on the free Lie algebra.
To define $\d$ the Leibniz identity,
or the third component of the Leibniz operad $\Leib(3)$,
was not used. Therefore
\begin{corollary}
For any binary product on $\g$,
although in general $\d\d\neq 0$,
a map $\d$ is well-defined
by the same manner as above.
\end{corollary}
\begin{proof}
Even if $\g$ is not Leibniz algebra,
if it has a binary product,
then $\s\Leib(2)$ acts on $\g[1]^{\ot 2}$.
Hence by the same manner as above
a map $\d$ is well-defined.
\end{proof}
From (\ref{fderirule}),
$H_{1}(F_{Lie}\g[1],\d)=\g_{Lie}$.
So we consider the chain complex
$$
\DR\g[1]:=
\Big(
F_{Lie}^{\bullet}\g[1]\overset{\d}{\longrightarrow}\g_{Lie}
\overset{0}{\longrightarrow}0
\Big),
$$
where $\g[1]\overset{\d}{\longrightarrow}\g_{Lie}$
is the augmentation.
\begin{theorem}
The total space $\DR\g[1]=\g_{Lie}\oplus F_{Lie}\g[1]$
becomes a dg-Lie algebra and the Leibniz bracket
on the Lie-Leibniz algebra $F_{Lie}\g[1]$
is the derived bracket.
Therefore the inclusion,
$$
\iota:\g[1]\to\g_{Lie}\oplus\F_{Lie}\g[1],
$$
is a derived bracket representation.
This representation is universal,
namely, an arbitrary derived
bracket representation of $\g$ is factors through $\iota$.
\end{theorem}
The derived bracket construction is the functor,
$\DC$, from the category of dg-Lie algebas
to the one of odd Leibniz algebras.
The theorem says that $\DR$ is the adjoint functor of $\DC$.
\begin{proof}
We denote by $\bar{x}:=p(x)$,
where $p:\g\to\g_{Lie}$ is the Liezation.\\
(A) The Lie algebra $\g_{Lie}$ acts on $\g[1]$ as follows.
\begin{equation}\label{defaction}
(\bar{x}_{1},x_{2}):=[x_{1},x_{2}],
\end{equation}
which is a representation of the Lie algebra.
Since $(\bar{x}_{1},-)$ is a linear map on $\g[1]$,
this operation
can be extended on $F_{Lie}\g[1]$ as a derivation
on the free Lie algebra.
Hence the semi-direct product plus the free Lie bracket,
$$
(\bar{x}_{1}\oplus\alpha_{1},\bar{x}_{2}\oplus\alpha_{2}):=
(\bar{x}_{1},\bar{x}_{2})\oplus(\bar{x}_{1},\alpha_{2})-
(\bar{x}_{2},\alpha_{1})+(\alpha_{1},\alpha_{2}),
$$
is a Lie bracket on $\g_{Lie}\oplus F_{Lie}\g[1]$.
One can easily see that $\g_{Lie}\oplus F_{Lie}\g[1]$
becomes a dg Lie algebra.
From (\ref{explid1}) and (\ref{defaction}),
we notice that
\begin{claim}\label{claiminproof}
The differential on $\g_{Lie}\oplus F_{Lie}\g[1]$
is generated from
$\d:\g[1]\longrightarrow\g_{Lie}$.
\end{claim}
\noindent(B)
On the universality.
Let $(\h,(.,.),D)$ be a dg Lie algebra
and let $f:\g[1]\to(\h,(D.,.))$
a derived bracket representation of $\g$.
We should prove that an arbitrary dg Lie algebra mapping
$$
\psi:\g_{Lie}\oplus F_{Lie}\g[1]\to(\h,(.,.),D)
$$
is factors through
$\iota:\g[1]\to\g_{Lie}\oplus F_{Lie}$
and $f=\psi\c\iota$.\\
(B1)
Since $F_{Lie}\g[1]$ is the free Lie algebra,
by its universality,
a Lie algebra morphism,
$\psi_{F}:F_{Lie}\g[1]\to\h$,
such that $\psi_{F}(x)=f(x)$ is uniquely well-defined,
where $x\in\g[1]$.\\
(B2) For any $\bar{x}\in\g_{Lie}$, we define a map
$\psi_{Lie}:\g_{Lie}\to\h$ as
$$
\psi_{Lie}(\bar{x}):=Df(x).
$$
We should check that $\psi_{Lie}$ is well-defined.
It suffices to show that $Df(I[1])=0$,
where $I[1]\subset\g[1]$ is the ideal
consisting of the symmetric brackets.
Because $f$ is a Leibniz homomorphism,
$f[x_{1},x_{2}]=(Df(x_{1}),f(x_{2}))$.
For any element of $I[1]$,
\begin{eqnarray*}
Df\big([x_{1},x_{2}]+[x_{2},x_{1}]\big)
&=&D(Df(x_{1}),f(x_{2}))+
D(Df(x_{2}),f(x_{1}))\\
&=&(Df(x_{1}),Df(x_{2}))+(Df(x_{2}),Df(x_{1}))\\
&=&(Df(x_{1}),Df(x_{2}))
-(Df(x_{1}),Df(x_{2}))=0.
\end{eqnarray*}
Therefore, $\psi_{Lie}$ is well-defined.\\
(B3)
We prove that
$\psi:=\psi_{Lie}\oplus\psi_{F}$
is a Lie algebra homomorphism.
Firstly,
$$
\psi(\bar{x}_{1},\bar{x}_{2})
=\psi_{Lie}\overline{[x_{1},x_{2}]}=
Df[x_{1},x_{2}]=(Df(x_{1}),Df(x_{2}))
=(\psi\bar{x}_{1},\psi\bar{x}_{2}).
$$
Secondly,
$$
\psi(\bar{x}_{1},x_{2})
=\psi_{F}[x_{1},x_{2}]=
f[x_{1},x_{2}]=(Df(x_{1}),f(x_{2}))
=(\psi(\bar{x}_{1}),\psi(x_{2}))
$$
and this implies that for any $\alpha\in F_{Lie}(\g[1])$,
$\psi(\bar{x},\alpha)=(\psi(\bar{x}),\psi(\alpha))$.\\
(B4)
Finally, we prove that $\psi$ is commutative with
the differentials, i.e., $\psi\d = D\psi$.
Thanks to the claim above,
it suffices to check the two cases of
$\psi\d(x)=D\psi(x)$ and
$\psi\d(\bar{x})=D\psi(\bar{x})$.
The first case is
$$
\psi\d(x)=\psi_{Lie}(\bar{x})=Df(x)=D\psi(x)
$$
and the second case is obvious,
because $\d\bar{x}=0$ and $DD=0$.
\end{proof}
Finally of this section, we observe
the second homology group, $H_{2}(F_{Lie}\g[1],\d)$.
The invariant bilinear form
in the category of Lie algebra
is a symmetric pairing $\langle.,.\rangle$
satisfying the invariant condition
\begin{equation}\label{classicalinvcon}
\langle[x,y],z\rangle=\langle x,[y,z]\rangle,
\end{equation}
where $[.,.]$ is a Lie bracket.
The universal invariant bilinear form on
a Lie algebra $\g$ is by definition the projection of
$\g\ot\g$ to $\Omega^{0}(\g)$, where
$$
\Omega^{0}(\g)
:=\g\ot\g/\{x\ot y-y\ot x,[x,y]\ot z-x\ot[y,z]\}.
$$
According to Kontsevich \cite{Kont},
$\Omega^{0}(\g)$ is regarded as the space
of formal functions (0-forms) over $\g$
as a formal Lie-manifold.
It is known that
(\ref{classicalinvcon}) (non-universal version)
is a consequence of a derived bracket construction
(cf. \cite{Kos0}, \cite{Roy}).
The universal version also comes from the derived
bracket construction, that is,
\begin{proposition}
If $\g$ is a Lie algebra, then
$H_{2}(F_{Lie}\g[1],\d)=\Omega^{0}(\g)$.
\end{proposition}
\begin{proof}
Because $\g[1]$ is an odd space,
$F^{2}_{Lie}\g[1]$ is the same as
the symmetric tensor space,
$S^{2}\g:=\g\ot\g/\{x\ot y-y\ot x\}$,
and because the bracket is Lie, $\d=0$ on $S^{2}\g$.
On the other hand,
$$
\d(y,(x,z))=([y,x],z)+(x,[y,z])\sim
-[x,y]\ot z+x\ot[y,z].
$$
Therefore, the identity of the corollary holds.
\end{proof}
Let us consider the Leibniz case.
In this case,
the universal symmetric bilinear form is not defined
on $S^{2}\g$, but on $\Ker_{2}\d$,
and the target space of the bilinear form is
$H_{2}(F_{Lie}\g[1],\d)=\Ker_{2}\d/\Im_{2}\d$.
From (\ref{fderirule2}),
the invariant condition has the following form.
\begin{equation}\label{invcou}
([x,y],z)+(y,[x,z])=(x,[y,z]+[z,y]),
\end{equation}
which is the same as the invariant condition
for Courant algebroids (see \cite{Kos3} for the details).
The meaning of (\ref{invcou}) is clear.
If $\L$ is a Lie subalgebra of the Leibniz algebra $\g$,
then $S^{2}\L$ is a subspace of $\Ker_{2}\d$
and then (\ref{invcou}) reduces to the classical
formula over $\L$.
Namely, (\ref{invcou}) is
the relation for the Lie subalgebras of $\g$.

\section{Anti-cyclic subcomplex}

The aim of this section is to describe how the complex
$(F_{Lie}\g[1],\d)$ relates with the (co)homology
complex of Leibniz algebra.
In 4.1 we will prove that $(F_{Lie}\g[1],\d)$
is a subcomplex of Leibniz homology
complex and $\d$ is the same
as the boundary map of Loday.
In 4.2 we will introduce the notion of anti-cyclic cochain
for Leibniz algebras by analogy with cyclic cochains
for associative algebras
and prove that the set of anti-cyclic cochains
is a subcomplex of the cohomology complex
of Loday-Phirashvili.

\subsection{Homology side}

Let $\g$ be a Leibniz algebra.
The complex over $\g$ computing the Leibniz homology
group is the tensor space
$\bar{T}\g[1]=\bigoplus_{n\in\mathbb{N}}\g[1]^{\ot n}$
with the boundary map,
\begin{equation}\label{defpaL}
\pa_{L}(x_{1},...,x_{n})
:=\sum_{1\le i<j\le n}(-1)^{i-1}(x_{1},...,x_{i}^{\vee},...,[x_{i},x_{j}],x_{j+1},...,x_{n}).
\end{equation}
The definition of $\pa_{L}$ is the left-version
of Loday's original formula.
We call $(\bar{T}\g[1],\pa_{L})$ a Loday complex.
The free Lie algebra $F_{Lie}\g[1]$ is regarded
as a subspace of the tensor space via the commutator,
\begin{equation}\label{freecommutator}
\{x_{1},...,x_{n}\}=x_{1}\ot\{x_{2},...,x_{n}\}-
(-1)^{n-1}\{x_{2},...,x_{n}\}\ot x_{1},
\end{equation}
where $\{x_{1},...,x_{n}\}$ is the right-normalized
bracket used in (\ref{explid1}).
\begin{theorem}
$\pa_{L}=\d$ on $F_{Lie}\g[1]$.
\end{theorem}
As a result, the free Lie algebra is a subcomplex
of Loday complex.
\begin{proof}
Obviously $\pa_{L}(x_{1},x_{2})=\d(x_{1},x_{2})$.
\begin{lemma}
$\pa_{L}(\{x_{2},...,x_{n}\}\ot x_{1})=(\pa_{L}\{x_{2},...,x_{n}\})\ot x_{1}$
\end{lemma}
\begin{proof}
From the defining equation of $\pa_{L}$,
we have
$$
\pa_{L}(\{x_{2},...,x_{n}\}\ot x_{1})
=(\pa_{L}\{x_{2},...,x_{n}\})\ot x_{1}+T
$$
where $T$ is the term which has $[-,x_{1}]$,
$$
T=\sum(,...,x_{i}^{\vee},...,)\ot[x_{i},x_{1}].
$$
We should prove $T=0$.
From (\ref{freecommutator}),
\begin{equation}\label{41p1}
\{x_{2},...,x_{n}\}\ot x_{1}
=x_{2}\ot\{x_{3},...,x_{n}\}\ot x_{1}
-(-1)^{n-2}\{x_{3},...,x_{n}\}\ot x_{2}\ot x_{1}.
\end{equation}
Therefore,
$[x_{2},x_{1}]$ appears in $T$ in two ways.
One is from the first term of (\ref{41p1})
$$
(-1)^{n-2}(-1)^{n-2}\{x_{3},...,x_{n}\}\ot[x_{2},x_{1}]
$$
and the other is from the second term
$$
-(-1)^{n-2}(-1)^{n-2}\{x_{3},...,x_{n}\}\ot[x_{2},x_{1}].
$$
Because the sign is reverse to each other,
the terms with $[x_{2},x_{1}]$ vanish.
By repeating the same discussion, we obtain $T=0$.
\end{proof}
From (\ref{defpaL}), it is easy to see through
\begin{lemma}
$$
\pa_{L}(x_{1}\ot\{x_{2},...,x_{n}\})=\\
\sum_{2\le i\le n}\{x_{2},...,[x_{1},x_{i}],...,x_{n}\}-
x_{1}\ot\pa_{L}\{x_{2},...,x_{n}\}.
$$
\end{lemma}
Therefore,
\begin{multline*}
\pa_{L}\{x_{1},...,x_{n}\}=\\
\sum_{2\le i\le n}\{x_{2},...,[x_{1},x_{i}],...,x_{n}\}-
x_{1}\ot\pa_{L}\{x_{2},...,x_{n}\}
-(-1)^{n-1}(\pa_{L}\{x_{2},...,x_{n}\})\ot x_{1}.
\end{multline*}
By assumption of induction,
$\pa_{L}\{x_{2},...,x_{n}\}=\d\{x_{2},...,x_{n}\}$.
Hence
\begin{eqnarray*}
\pa_{L}\{x_{1},...,x_{n}\}&=&
\sum_{2\le i\le n}\{x_{2},...,[x_{1},x_{i}],...,x_{n}\}
-\{x_{1},\d\{x_{2},...,x_{n}\}\}\\
&=&\d\{x_{1},...,x_{n}\}
\end{eqnarray*}
The proof is completed.
\end{proof}
In the following we denote by
$HA_{\bullet-1}(\g):=H_{\bullet}(F_{Lie}\g[1],\d)$.
Hence $HA_{0}(\g)=\g_{Lie}$ and if $\g$ is Lie,
then $HA_{1}(\g)=\Omega^{0}(\g)$.

\subsection{Cohomology side}

We recall the cohomology complex for
Leibniz algebra \cite{LP}.
Let $\g$ be a Leibniz algebra
and $M$ a $\g$-module or representation of $\g$.
The cochain complex which computes the cohomology group
of $\g$ with coefficients in $M$ is
$$
LP(\g,M)
:=\Hom_{\mathbb{K}}(\g[1],M[1])
$$
equipped with a differential defined by
\begin{multline}\label{df}
(d_{LP}f)(x_{1},...,x_{n+1})
:=
[f(x_{1},...,x_{n-1}),x_{n+1}]+
\sum_{i=1}^{n}
(-1)^{i+n}
[x_{i},f(x_{1},...,x_{i}^{\vee},...,x_{n+1})]\\
-\sum_{i<j\le n+1}(-1)^{i+n}
f(x_{1},...,x_{i}^{\vee},...,[x_{i},x_{j}],x_{j+1},...,x_{n+1}),
\end{multline}
where $f\in LP^{n}(\g,M)$ and $|f|:=n-1$.
This definition of the derivation is
the left-version of the original formula
introduced in \cite{LP}.
\medskip\\
\indent
Let $LP^{n}(\g)$ be the space of $n+1$-linear functions
on the tensor space $\bar{T}\g[1]$,
$$
LP^{n}(\g):=\Hom(\g[1]^{\ot n+1},\mathbb{K}).
$$
The differential $d_{LP}$ can be extended
on $LP^{\bullet}(\g)$ by the following manner,
\begin{multline}\label{defblp}
(b_{LP}\tilde{f})(x_{1},...,x_{n+1},x_{n+2}):=
(-1)^{n}
\tilde{f}(x_{1},...,x_{n},[x_{n+1},x_{n+2}]+[x_{n+2},x_{n+1}])+\\
\sum_{\substack{i<j \\ i\le n}}(-1)^{i-1}\tilde{f}(x_{1},...,x_{i}^{\vee},...,
[x_{i},x_{j}],x_{j+1},...,x_{n+2}),
\end{multline}
where $\tilde{f}\in LP^{n}(\g)$.
When $\g$ is finite dimensional and
$f\in LP^{n}(\g,\g^{*})$, if
we put
$$
\ti{f}(x_{1},...,x_{n+1})
:=\omega(f(x_{1},...,x_{n}),x_{n+1}),
$$
then $b_{LP}\ti{f}=(-1)^{n}\widetilde{d_{LP}f}$,
where $\omega$ is the canonical structure in (\ref{defomega}).\\
\indent
Now we define the notion of anti-cyclic cochain.
Before giving a general definition, let us observe
the elementary case.
Let $\g$ be a finite dimensional Leibniz algebra.
Consider an Abelian extension of $\g$ by $\g^{*}$,
$$
\begin{CD}
0@>>>\g^{*}@>>>\g\oplus\g^{*}@>>>\g @>>>0.
\end{CD}
$$
In general the Leibniz bracket on the middle position
has the following form,
$$
[x_{1}\oplus a_{1},x_{2}\oplus a_{2}]=[x_{1},x_{2}]
\oplus[x_{1},a_{2}]+[a_{1},x_{2}]+H(x_{1},x_{2}),
$$
where $x_{1},x_{2}\in\g$,
$a_{1},a_{2}\in\g^{*}$
and $H$ is a 2-cocycle in $LP^{2}(\g,\g^{*})$.
This bracket satisfies the anti-invariant condition
(\ref{A1})-(\ref{A2}) if and only if
\begin{eqnarray}
\label{H1}\widetilde{H}(x_{1},x_{2},x_{3})&=&
\widetilde{H}(x_{1},x_{3},x_{2}),\\
\label{H2}\oint\widetilde{H}(x_{1},x_{2},x_{3})&=&0,
\end{eqnarray}
where $\oint$ is the cyclic permutation for $x_{1},x_{2},x_{3}$.
We notice that the symmetry that $\widetilde{H}$
satisfies is the same as the one of the Lie bracket
$(x_{1},(x_{2},x_{3}))$,
where $|x_{i}|=odd$ for any $i\in\{1,2,3\}$.
This observation leads us to the following definition.
In the following we denote by
$\epsilon:F_{Lie}\g[1]\hookrightarrow\bar{T}\g[1]$
the embedding of the free Lie algebra.
\begin{definition}
Let $\g$ be a Leibniz algebra not necessarily finite
dimensional.
An $n$-cochain $A(x_{1},...,x_{n+1})\in LP^{n}(\g)$
is called an anti-cyclic cochain,
or shortly ac-cochain, if
\begin{equation}\label{defA1}
A(x_{1},...,x_{n+1})=\frac{1}{n+1}A\epsilon\{x_{1},...,x_{n+1}\},
\end{equation}
where $A\epsilon$ is the pull-back of $A$ by
$\epsilon$.
\end{definition}
As a result, $A(x_{1},...,x_{n+1})$
satisfies the same symmetry as
$\{x_{1},...,x_{n+1}\}$.
For example, if $A$ is an ac 1-cochain,
then it is a symmetric tensor,
$A(x_{1},x_{2})=A(x_{2},x_{1})$, because
$$
A(x_{1},x_{2})
=\frac{1}{2}A\epsilon\{x_{1},x_{2}\}
=\frac{1}{2}A\epsilon\{x_{2},x_{1}\}
=A(x_{2},x_{1}).
$$
If $A$ is ac 2, then it satisfies (\ref{H1})-(\ref{H2})
above. If $A$ is ac 3,
it has more complicated symmetry,
\begin{eqnarray*}
A(x_{1},x_{2},x_{3},x_{4})&=&A(x_{1},x_{2},x_{4},x_{3}),\\
\oint_{234}A(x_{1},x_{2},x_{3},x_{4})&=&0,\\
A(x_{1},x_{2},x_{3},x_{4})+A(x_{2},x_{1},x_{3},x_{4})
&=&
-A(x_{3},x_{4},x_{1},x_{2})-A(x_{4},x_{3},x_{1},x_{2}),
\end{eqnarray*}
where $\oint_{234}$ is the cyclic permutation
for $x_{2},x_{3},x_{4}$
and the last identity comes from
$\{x_{1},x_{2},x_{3},x_{4}\}
=\{(x_{1},x_{2}),(x_{3},x_{4})\}
-\{x_{2},x_{1},x_{3},x_{4}\}$.
Let us denote by $ALP^{\bullet}(\g)$
the space of ac-cochains.
\begin{lemma}[Implicit definition]\label{impdef}
$A$ is in $ALP^{n}(\g)$
if and only if
there exists a linear function,
$A^{\p}$, on $F_{Lie}^{n+1}\g[1]$ and
\begin{equation}\label{defA2}
A(x_{1},...,x_{n+1})=A^{\p}\{x_{1},...,x_{n+1}\}.
\end{equation}
\end{lemma}
\begin{proof}
If $A$ is an ac-cochain, $A^{\p}:=\frac{1}{n+1}A\epsilon$.
The converse is also easy (See Appendix).
\end{proof}

\begin{theorem}
$ALP^{\bullet}(\g)$ is a subcomplex of
$(LP^{\bullet}(\g),b_{LP})$.
\end{theorem}
\begin{proof}
Suppose that $A(x_{1},...,x_{n+1})$ is an ac n-cochain.
From (\ref{defblp}) and the assumption,
\begin{multline*}
(b_{LP}A)(x_{1},...,x_{n+1},x_{n+2})=
(-1)^{n}
A(x_{1},...,x_{n},[x_{n+1},x_{n+2}]+[x_{n+2},x_{n+1}])+\\
\sum_{\substack{i<j \\ i\le n}}(-1)^{i-1}
A(x_{1},...,x_{i}^{\vee},...,
[x_{i},x_{j}],x_{j+1},...,x_{n+2})=\\
=(-1)^{n}
A^{\p}\{x_{1},...,x_{n},[x_{n+1},x_{n+2}]+[x_{n+2},x_{n+1}]\}+\\
\sum_{\substack{i<j \\ i\le n}}(-1)^{i-1}
A^{\p}\{x_{1},...,x_{i}^{\vee},...,
[x_{i},x_{j}],x_{j+1},...,x_{n+2}\},
\end{multline*}
on the other hand, from (\ref{explid2}),
the right-hand side is equal to
$(A^{\p}\d)\{x_{1},...,x_{n+2}\}$.
Hence we obtain
$$
(b_{LP}A)(x_{1},...,x_{n+2})
=(A^{\p}\d)\{x_{1},...,x_{n+2}\},
$$
which yields the theorem, i.e.,
$(b_{LP}A)^{\p}=(A^{\p}\d)$.
\end{proof}
Denote by $HA^{\bullet}(\g):=H^{\bullet}(ALP(\g),b_{LP})$
the cohomology group of anti-cyclic cochains.
The space of ac 0-cochains is equal to the dual space
$\g^{*}:=\Hom(\g,\mathbb{K})$.
When $A$ is an ac 0-cochain, then
$b_{LP}A=0$ if and only if $A=0$ on the ideal $I$.
Hence $HA^{0}(\g)=I^{\bot}\cong\g_{Lie}^{*}=(HA_{0}\g)^{*}$.
When $\g$ is Lie,
if $A$ is an ac 1-cocycle, then
$$
A([x,y],z)+A(y,[x,z])=A(z,[x,y])-A([z,x],y)=0.
$$
Hence $HA^{1}(\g)=(HA_{1}\g)^{*}=(\Omega^{0}\g)^{*}$.
We here prove a classical theorem.
Let $\g$ be a finite dimensional Leibniz algebra.
We consider a subclass of Abelian extensions of
$\g$ by $\g^{*}$ such that\\
(i) the Leibniz algebra
of the middle position, $\g\oplus\g^{*}$,
satisfies (\ref{A1})-(\ref{A2}) with respect to $\omega$,\\
(ii) the isomorphisms between extensions
preserve $\omega$.
\begin{theorem}
Such extensions are classified into $HA^{2}(\g)$.
\end{theorem}
\begin{proof}
In general, an isomorphism between
Abelian extensions is given by $e^{\tau}:=1+\tau$,
where $\tau:\g\to\g^{*}$.
This preserves $\omega$ if and only if $\tilde{\tau}$
is a symmetric tensor or ac $1$-cochain.
\end{proof}
\section{Tensor expression}
In this section we study a tensor expression
of anti-cyclic cochains.
In the following suppose that $\g$
is a finite dimensional Leibniz algebra.
Let $e_{1},...,e_{\dim\g}$ be a base of $\g$.
The degree of $e_{i}$ is $|e_{i}|=-1$ for each $i$.
If $A$ is an ac 2-cochain on $\g$, then
$A(e_{i},e_{j},e_{k})=A_{ijk}$.
Hence the cochain is expressed as
$A=A_{ijk}e^{i}\ot e^{j}\ot e^{k}$,
where $e^{i}$ is the dual base of $e_{i}$.
The coefficient part, $A_{ijk}$, satisfies
\begin{eqnarray*}
A_{ijk}&=&A_{ikj},\\
A_{ijk}+cyclic&=&0,
\end{eqnarray*}
which is the symmetry that the normalized
Lie bracket $(e_{i},(e_{j},e_{k}))$ satisfies.
Hence the symmetry of
the tensor part, $e^{i}\ot e^{j}\ot e^{k}$,
should be the dual of the one of $(e_{i},(e_{j},e_{k}))$.
We denote such a tensor by $\{e^{i},e^{j},e^{k}\}_{*}$
and call the bracket $\{,...,\}_{*}$ a dual Lie bracket.
In general, the dual Lie bracket is defined as follows
\begin{definition}[dual Lie brackets]
\begin{equation}\label{defduallie}
\{x^{1},...,x^{n}\}_{*}:=x^{1}\ot\{x^{2},...,x^{n}\}_{*}
-(-1)^{n-1}
x^{n}\ot\{x^{1},...,x^{n-1}\}_{*},
\end{equation}
where we put $|x^{i}|:=odd$ or $+1$.
In particular, $\{x^{1}\}_{*}=x^{1}$.
\end{definition}
For example,
$\{x^{1},x^{2}\}_{*}$ is equal to the symmetric tensor
$\{x^{1},x^{2}\}_{*}=x^{1}\ot x^{2}+x^{2}\ot x^{1}$,
the 3-ary bracket is
$$
\{x^{1},x^{2},x^{3}\}_{*}=
x^{1}\ot x^{2}\ot x^{3}+
x^{1}\ot x^{3}\ot x^{2}-
x^{3}\ot x^{1}\ot x^{2}-
x^{3}\ot x^{2}\ot x^{1}.
$$
From (\ref{defduallie})
the total cyclic summation of the dual-Lie bracket is zero.
$$
\oint\{x^{1},...,x^{n}\}_{*}=0,
$$
where $\oint$ is the cyclic permutation for all variables.
An ac n-cochain on $\g$ is expressed by using the
dual Lie bracket as follows.
$$
A=\frac{1}{n+1}
\sum A_{i_{1}...i_{n+1}}\{e^{i_{1}},...,e^{i_{n+1}}\}_{*}.
$$
\begin{definition}[contraction]
If $f$ is a linear function on $\g$,
$$
i_{f}\{x^{1},...,x^{n}\}_{*}
:=f(x^{1})\{x^{2},...,x^{n}\}_{*}
-(-1)^{n-1}f(x^{n})\{x^{1},...,x^{n-1}\}_{*}.
$$
\end{definition}
Our interesting space is not $\g$
but the double space $\g\oplus\g^{*}$.
By analogy with the Lie algebra case,
the notion of Cartan 3-form
is defined by $C(x,y,z):=\omega([x,y],z)$
on $\g\oplus\g^{*}$,
where $[.,.]$ is the Leibniz bracket of $\g\ltimes\g^{*}$.
The structure constant of the Leibniz bracket
is defined by using the Cartan 3-form
$$
C_{ij}^{k}:=C(e_{i},e_{j},e^{k})
=\omega([e_{i},e_{j}],e^{k}).
$$
Denote
$\mu_{Leib}:=C_{ij}^{k}\{e^{i},e^{j},e_{k}\}_{*}$.
Then for any linear functions $f_{1},f_{2}$
on $\g\oplus\g^{*}$,
the Leibniz bracket $[f_{1},f_{2}]$
is computed by
$$
[f_{1},f_{2}]=i_{f_{2}}i_{f_{1}}\mu_{Leib}.
$$
For $\widetilde{H}$ in (\ref{H1})-(\ref{H2}),
denote $H_{ijk}:=\widetilde{H}(e_{i},e_{j},e_{k})$.
Then
$\widetilde{H}=
\frac{1}{3}H_{ijk}\{e^{i},e^{j},e^{k}\}_{*}$.
Therefore, the total structure with
the twisting term $\widetilde{H}$ is
expressed as follows.
$$
\theta_{Leib}:=C_{ij}^{k}\{e^{i},e^{j},e_{k}\}_{*}+
\frac{1}{3}H_{ijk}\{e^{i},e^{j},e^{k}\}_{*}.
$$
Classical structures for Lie algebras
are expressed by using the wedge product,
$$
\theta_{Lie}=\frac{1}{2}
C_{ij}^{k}e^{i}\wedge e^{j}\wedge e_{k}
+\frac{1}{6}H_{ijk}e^{i}\wedge e^{j}\wedge e^{k}.
$$
On the other hand in the Leibniz world,
the structure tensors are expressed by
the dual Lie bracket in stead of the wedge product.

\begin{center}
\textbf{Appendix --Proof of Lemma \ref{impdef}--}
\end{center}
Denote by $\{,\}$ the map of higher bracketing
$$
\{,\}:x_{1}\ot\cdots\ot x_{n}\mapsto\{x_{1},...,x_{n}\}.
$$
Then the following identity holds.
\begin{equation}\label{AX1}
\{,\}\epsilon\{,\}=n\{,\},
\end{equation}
where $n$ is the length of word.
The lemma is a consequence of this identity,
hence we prove (\ref{AX1}).\\
\indent
When $n=1,2$ the identity obviously holds.
By the Jacobi identity,
in general, the normalized bracket satisfies
\begin{equation}\label{AX2}
\{,...,\epsilon\{,...,\},...,x_{f}\}
=\{,...,\{,...,\},...,x_{f}\},
\end{equation}
where $x_{f}$ is the fixed variable
which lies the most right position.
For example,
$\epsilon\{x_{1},x_{2}\}=x_{1}\ot x_{2}+x_{2}\ot x_{1}$
and
$$
\{\epsilon\{x_{1},x_{2}\},x_{3}\}=
\{x_{1},x_{2},x_{3}\}+\{x_{2},x_{1},x_{3}\}=
\{\{x_{1},x_{2}\},x_{3}\}.
$$
From the definition,
\begin{equation}\label{AX3}
\epsilon\{x_{1},...,x_{n+1}\}=x_{1}\ot\epsilon\{x_{2},...,x_{n+1}\}
-(-1)^{n}\epsilon\{x_{2},...,x_{n+1}\}\ot x_{1}.
\end{equation}
Applying $\{,\}$ to (\ref{AX3}), we obtain
$$
\{,\}\epsilon\{x_{1},...,x_{n+1}\}=
\{x_{1},\{\}\epsilon\{x_{2},...,x_{n+1}\}\}
-(-1)^{n}\{\epsilon\{x_{2},...,x_{n+1}\},x_{1}\}.
$$
By the assumption of induction, the first term is
equal to $n\{x_{1},...,x_{n}\}$
and by (\ref{AX2}) the second term is equal to
$\{x_{1},...,x_{n+1}\}$.
Hence
$$
\{,\}\epsilon\{x_{1},...,x_{n+1}\}
=(n+1)\{x_{1},...,x_{n+1}\}.
$$
The proof is completed.

\begin{verbatim}
Kyousuke UCHINO
email:kuchinon@gmail.com
\end{verbatim}
\end{document}